\documentclass[a4paper,12pt]{article}
\usepackage[T2A]{fontenc}
\usepackage[utf8]{inputenc} 
\usepackage[english]{babel}
\usepackage{amsmath,amsthm,amssymb}
\usepackage{xcolor}
\usepackage{hyperref}
\usepackage{currfile}
\usepackage{navigator}
\IfFileExists{\currfilename}{\embeddedfile{sourcefile}{\currfilename}}{}

\newtheorem{theorem}{Theorem}
\newtheorem{lemma}{Lemma}
\newtheorem{proposition}{Proposition}
\newtheorem{corollary}{Corollary}
\theoremstyle{remark}
\newtheorem{remark}{\bf Remark}

\makeatletter
\def\@seccntformat#1{\csname the#1\endcsname.\ } 
\def\@biblabel#1{#1.} 
\makeatother

\newcommand\vc[1]{\bar{#1}} 
\newcommand\GF[1]{\mathbb{F}_{#1}} 

\newcommand\Quot[1]{\mbox{$\begin{pmatrix}#1\end{pmatrix}$}}

\title{Low-degree colorings without non-essential arguments}
\author{Denis S. Krotov%
\thanks{D.S.K. is with the Sobolev Institute of Mathematics. E-mail: dk@ieee.org}
}
\date{}

\begin{document}

\maketitle
\begin{abstract}
For the Hamming graph $H(n,q)$, where $q$ is a constant prime power and $n$ grows, we construct perfect colorings without non-essential arguments such that $n$ depends exponentially on the off-diagonal part of the quotient matrix. In particular, we construct unbalanced functions (in the case $q=2$, Boolean functions) such that the number of essential arguments depends exponentially on the degree of the function.

Another framing of the main construction is that it gives a sequence of very regular functions (perfect colorings) whose number of essential arguments grows exponentially with respect to sensitivity. A classical result that connects sensitivity with the number of essential arguments in the binary ($q=2$) case is the Simon--Wegener theorem. We prove an analog of this theorem for an arbitrary~$q$.


\end{abstract}


\section{Introduction}

The central result of this paper is the construction of unbalanced Boolean functions whose number of essential arguments grows exponentially with respect to their degree. Balanced functions with similar properties were already known
\cite{Tarannikov:2002:survey}, \cite{CHS:2020}, and the separate consideration of the unbalanced case is natural.
Indeed, the behavior of some characteristics of balanced and unbalanced Boolean functions can be different;
for example, the correlation immunity of a balanced Boolean function
in $n$ variables can reach $n-1$, while in the unbalanced case
it cannot exceed $2n/3-1$~\cite{FDF:CorImm}.
We develop the construction in the more general setting of functions of $n$ $q$-valued arguments, treated as colorings of the Hamming graph $H(n,q)$.
\medskip

An arbitrary surjective function 
from the vertex set
of a graph~$G$ onto a finite set~$K$ 
(of \emph{colors})
of cardinality~$k$ is called 
a (vertex) \emph{coloring}, 
or \emph{$k$-coloring}, of~$G$. 
A $k$-coloring is called \emph{perfect}
if there is a $k$-by-$k$ 
matrix $\{S_{i,j}\}_{i,j\in K}$
(the \emph{quotient} matrix)
such that every vertex of color~$i$
has exactly~$S_{i,j}$ neighbors
of color~$j$.
The \emph{Hamming graph} $H(n,q)$ 
is a graph on the set 
of the words of length~$n$
over an alphabet~$\Sigma$ of size~$q$,
two words being adjacent if and only if 
they differ in exactly one position.
Colorings of the Hamming graph
can be considered as functions
in $n$ $\Sigma$-valued arguments.
Such a function~$f$ is said 
to depend 
on the $i$th argument 
\emph{essentially} (\emph{non-essentially}), 
$i\in\{0,\ldots,n-1\}$,
if there are (there are no)
two vertices~$\vc{x}$, $\vc{y}$
of $H(n,q)$ differing 
in only the $i$th position
such that $f(\vc{x})\ne f(\vc{y})$.
A $\{0,1\}$-valued function 
on the vertices of $H(n,2)$
is called a \emph{Boolean function}.

It is well known that if there exists
a perfect coloring of~$H(n,q)$ with quotient matrix~$S$,
then one can construct a perfect coloring
of~$H(n+1,q)$  
with quotient matrix~$S+(q-1)I$
($I$ is the identity matrix),
just adding a non-essential argument.

In this context,
it is natural to consider the following
questions.

\begin{itemize}
    \item[{[$n_0$]}] For a given matrix~$S$ and a fixed~$q$,
    what is the minimum value of~$n$ (call it~$n_0(S)$)
   such that  there exists a perfect coloring of~$H(n,q)$
   with quotient matrix~$S$, up to the main diagonal. In particular does there exist such~$n$?
\end{itemize}

A criterion of existence of~$n_0$ for $2$-colorings
is known for prime~$q$
($q=2$: \cite{FDF:PerfCol}, $q>2$: \cite{BKMTV});
the exact value of~$n_0$ even for $2$-colorings of~$H(n,2)$
is unknown in many cases
(for example, it is not known if there exists a perfect coloring
of~$H(24,2)$ with quotient matrix $\Quot{1&23\\9&15}$, see more examples in the small-parameters table~\cite{KKM:tables}).

\begin{itemize}
    \item[{[$n_1$]}]
    For a given matrix $S$ and a fixed $q$,
    assuming the existence of~$n_0(S)$,
    what is the maximum value of~$n$ (call it~$n_1(S)$)
   such that there exists a perfect coloring of~$H(n,q)$
   with quotient matrix~$S$, up to the main diagonal, without
   non-essential arguments?

   Note that the existence of such~$n_1$ 
   for $q=2$ follows
   from the Simon--Wegener theorem~\cite{Wegener}.
   In Section~\ref{s:sensitivity}, we prove finiteness for every $q\ge 2$. Removing all inessential arguments preserves the off-diagonal entries of the quotient matrix,
   so existence of $n_0(S)$ also ensures that the set of dimensions in the definition of $n_1(S)$ is nonempty.
\end{itemize}

Results in this direction are known for
balanced (with symmetric quotient matrix) $2$-colorings
of $H(n,2)$.
{For the symmetric case, Tarannikov~\cite{Tarannikov:2002:survey,Tarannikov:2002} obtained the following bounds, mentioned in~\cite{FDF:PerfCol}. The lower bound also follows from the construction in~\cite{CHS:2020}.}
\begin{proposition}\label{p:Tar}
    For the quotient matrix
    $S=\left(\begin{array}{cc} a&b\\b&a\end{array}\right)$ of
    a perfect $2$-coloring of $H(n,2)$, 
    we have
    $$ 3\cdot 2^{b-1} -2 \le n_1(S) \le b\,2^{b-1}. $$
\end{proposition}
Our goal is to generalize this result to unbalanced perfect colorings of $H(n,2)$ and to perfect colorings of $H(n,q)$ for arbitrary prime powers~$q$ by showing that $n_1$ can depend exponentially on the off-diagonal part of the quotient matrix. The construction is also of interest in the broader context of unbalanced functions whose number of essential variables grows exponentially with their degree (defined and discussed in Section~\ref{s:eigdeg}).

The construction also shows that the number of essential arguments of perfect colorings can grow exponentially with their sensitivity. This connects our results with the Simon--Wegener theorem. In Section~\ref{s:sensitivity}, we prove its $q$-ary analogue for every integer $q\ge2$, without a prime-power assumption.

The main results of this paper are the following:
Theorem~\ref{th:rec} and Corollary~\ref{c:constr} provide a recursive construction of perfect colorings of $H(n,q)$, where $q$ is a prime power, without non-essential arguments with quotient matrices
whose off-diagonal part grows as $O(\log n)$.
Theorem~\ref{th:bc} gives a concrete sequence
of such colorings for the case of two colors.
Corollary~\ref{c:bu} establishes the existence of
$\{0,1\}$-valued functions with any prescribed density $\rho=r/q^i$ of ones, whose number of essential arguments grows exponentially with respect to the degree.
Theorem~\ref{th:sensitivity} is a $q$-ary
version of the Simon--Wegener theorem that shows that the minimal sensitivity of a coloring of $H(n,q)$ with all arguments essential grows as
$ \frac{q-1}{2}\log_q n\,(1+o_q(1))$.
Corollary~\ref{c:perfect-sensitivity}
derives similar evaluations for perfect colorings (with different leading coefficients in the lower and upper bounds, i.e, the exact asymptotics remains unsolved in general).

In the following subsection, we continue the introductory part
by presenting background material about eigenvalues and degrees of functions on~$H(n,q)$.
\subsection{Eigenvalues, degrees, densities}\label{s:eigdeg}

The eigenvalues of a graph are 
the eigenvalues 
of its adjacency matrix.
Let $G$ be a graph with vertex set $V$ and let $\lambda$ be an eigenvalue of $G$. The set of neighbors of a vertex $x$ is denoted by $N(x)$.
A function $f:V\longrightarrow{\mathbb{R}}$ is called a {\em $\lambda$-eigenfunction} of $G$ if $f\not\equiv 0$ and the equality
\begin{equation}\label{Eq:Eigenfunction}
\lambda\cdot f(x)=\sum_{y\in{N(x)}}f(y)
\end{equation}
holds for any vertex $x\in V$. The set of functions $f:V\longrightarrow{\mathbb{R}}$ satisfying (\ref{Eq:Eigenfunction}) for any vertex $x\in V$ is called a 
{\em $\lambda$-eigenspace} of $G$.

The Hamming graph $H(n,q)$ has $n+1$ distinct eigenvalues $\lambda_i(n,q)=n(q-1)-q\cdot i$, where $0\leq i\leq n$.
Denote by $U_{i}(n,q)$ the $\lambda_{i}(n,q)$-eigenspace of $H(n,q)$. 

Every real-valued function $f$ 
on the vertices of $H(n,q)$ is uniquely 
represented as the sum
$$
f = \sum_{i=0}^n  \phi_i,
$$
where $\phi_i\in U_i(n,q)$. In fact, $\phi_i$ is the projection of~$f$ onto the eigenspace $U_{i}(n,q)$.
The maximum~$i$ for which $\phi_i\not\equiv 0$
is called the \emph{degree} of~$f$ 
(the degree of the constantly zero 
function is assumed to be~$0$).
It is not difficult to see that the degree
of a function~$f$ 
on the vertices of $H(n,2)$
(in particular, a Boolean function)
is the minimum degree 
of the polynomial
representation of~$f$,
where the vertex set of $H(n,2)$
is $\Sigma^n=\{a,b\}^n$ 
for any distinct real 
values~$a$ and~$b$ (typically, 
$1$ and~$-1$, or $0$ and~$1$; 
note that replacing $\Sigma$ 
does not affect the degree
of a function).
The \emph{degree} of a coloring~$f$ is
the maximum degree 
of the \emph{characteristic functions}
$\delta_i(f(\cdot))$ of its colors~$i$,
where $\delta_i(j) = 1$ if $j=i$ 
and $\delta_i(j) =0$ otherwise. {\color{black}The number of essential arguments of a Boolean function of degree~$d$ is at most $d2^{d-1}$ by Nisan and Szegedy~\cite{NisanSzegedy:1994}. Chiarelli, Hatami, and Saks~\cite{CHS:2020} improved this bound to $O(2^d)$, and Wellens~\cite{Wellens:2022} further improved the constant.} {\color{black}For $q\ge3$, Valyuzhenich~\cite{Valyuzhenich:2026} proved that a $\{0,1\}$-valued function of degree~$d$ on $H(n,q)$ has at most $dq^{d+1}/(4(q-1))$ essential arguments, with further bounds for functions with restricted spectral support, including perfect $2$-colorings. More recently, he obtained the bound $O_q(q^d)$~\cite{Valyuzhenich:2026:juntas}. Related bounds for two- and three-valued functions, using several notions of degree, were obtained by Potapov~\cite{Potapov:2025}. Here and throughout, degree means the Hamming spectral degree defined above; for $q>2$, it can differ from the minimum degree of a polynomial representation over a numerical alphabet.}

\begin{lemma}[{see e.g. \cite{KroPot:CRC&EP}}]\label{l:eig}
For a perfect coloring of a graph, 
the characteristic {function}
of each of the colors
is the sum of eigenfunctions 
of the graph
corresponding to eigenvalues 
of the quotient matrix.
Moreover, every eigenvalue 
of the quotient matrix 
is an eigenvalue of the graph.
\end{lemma}
\begin{corollary}
\label{l:union}
    If a perfect $k$-coloring of a connected $r$-regular graph
has quotient matrix with only two eigenvalues,
then unifying any two colors
results in a perfect $(k-1)$-coloring.
Conversely, if there are a perfect $k$-coloring~$f$
and a perfect $2$-coloring~$g$
of the same regular graph, their quotient 
matrices have only two eigenvalues~$r$ and~$\theta$,
and the coloring 
$h(\vc{x})=(f(\vc{x}),g(\vc{x}))$
is a $(k+1)$-coloring, then it is a perfect
coloring with two eigenvalues~$r$ and~$\theta$.
\end{corollary}

For a coloring of a finite graph, 
we define the 
\emph{density}~$\rho(i)$
of each color~$i$ as the proportion
of the vertices of color~$i$
among all vertices  of the graph.
The \emph{density vector} of 
a coloring is the list of densities 
of all colors;
it has sum~$1$.
For a perfect coloring of a connected graph,
the density vector is uniquely determined
by the quotient matrix~$\{S_{i,j}\}_{i,j}$,
because $\rho_i S_{i,j} = \rho_j S_{j,i}$
by double-counting the edges between 
color-$i$ and color-$j$ vertices.

\section{Main construction}\label{s:mainconstr}

We start with defining 
and proving the existence
of two main ingredients
of the construction.

\subsection{Uniform collections of colorings}
We say that a collection
$(C_0$, \ldots, $C_{M-1})$ 
of colorings of~$H(n,q)$
is \emph{uniform}
if the multiset of colors
\begin{equation}\label{eq:uniform}
  \{C_0(\vc{x}),\ldots,C_{M-1}(\vc{x})\}  
\end{equation}
does not depend on the vertex~$\vc{x}$ 
of~$H(n,q)$.
The following lemma is straightforward.

\begin{lemma}\label{l:dens}
For a uniform collection
$(C_0$, \ldots, $C_{M-1})$
of perfect colorings 
with the same quotient matrix,
the multiplicity
of each color~$i$
in the multiset~\eqref{eq:uniform}
equals~$\rho_i M$, where~$\rho_i$
is the density of the color~$i$ in~$C_0$
(as well as in any~$C_j$,
$j=0,\ldots,M-1$).
\end{lemma}
One can easily construct 
a uniform collection 
from every perfect coloring 
of the Hamming graph.
\begin{lemma}\label{l:uniform}
For every 
perfect coloring~$C$
of $H(n,q)$, 
there is a 
uniform collection
of $q^n$ perfect 
colorings with 
the same quotient matrix.
\end{lemma}
\begin{proof} 
Let $(\Sigma,+)$ be an arbitrary 
abelian group of order~$q$,
$(\Sigma^n,+)$ 
the direct product of~$n$ copies of $(\Sigma,+)$ 
($+$ acts component-wise on $\Sigma^n$). 
And let $H(n,q)$ be defined
on the vertex set $\Sigma^n$.
The $M:=q^n$ perfect colorings
$C_{\vc{z}}$, ${\vc{z}} \in \Sigma^n$,
where
$$ C_{\vc{z}}(\vc{x}):=C(\vc{x} - \vc{z}), $$
form a uniform collection, because 
$\{C_{\vc{z}}(\vc{x})\}_{\vc{z}\in \Sigma^n} $
equals
$
\{C(\vc{y})\}_{\vc{y}\in \Sigma^n}$ and hence does not depend on~$\vc{x}$. On the other hand, each $C_{\vc{z}}$
is a perfect coloring with the same quotient matrix as for~$C$, because $\vc{x}\to \vc{x} - \vc{z}$ 
is an automorphism 
of the Hamming 
graph~$H(n,q)$.
\end{proof}
In many cases,
it is possible
to construct
a uniform collection 
of size smaller than~$q^n$.
For example,
if the coloring~$C$
has $t$ \emph{periods}
$\vc{v}$ such that
$C(\vc{x}+\vc{v})\equiv C(\vc{x})$,
then in the 
uniform collection
constructed in the proof,
every distinct translate of the coloring is
repeated $t$ times. {Keeping} only one representative of each distinct translate reduces the size of the collection to~$q^n/t$.
Another obvious case is when the quotient matrix
is symmetric with respect to the colors,
for example, if $C$ is a perfect
$3$-coloring of $H(4,3)$ with quotient matrix
$$
\begin{pmatrix}
0&{4}&{4}\\{4}&0&{4}\\{4} & {4} & 0
\end{pmatrix},
$$
then a uniform $3$-collection 
$(C,\pi C, \pi \pi C)$ can be obtained from~$C$ by permuting the colors with the cyclic permutation~$\pi$.

\subsection{Partitions into RM-like codes}

As one of the ingredients
for our construction,
we need a perfect coloring
of $H(M,q)$
with special parameters,
whose existence 
(with some restrictions
on the parameters $q$ and $M$)
is guaranteed by the following lemma.

\begin{lemma}\label{l:RM}
For every prime power~$q$ and $M=q^s$,
where $s$ is a positive integer,
there is a perfect 
$M q$-coloring~$E$ of $H(M,q)$ with quotient matrix $T=(T_{i,j})$,
where 
\begin{equation}\label{eq:RM}
T_{i,j}=\begin{cases}
0, & \mbox{if $i\equiv j\bmod q $} \\
1, & \mbox{if $i\not\equiv j\bmod q $} 
\end{cases}.
\end{equation} 
\end{lemma}
\begin{proof}
Let $\Sigma$ be the finite field $\GF{q}$, and let
$\vc{\alpha}_0$, \ldots,
$\vc{\alpha}_{M-1}$
be all elements of~$\GF{q}^s$.
Define
$$
D(x_0,\ldots,x_{M-1})=
(\sum_i x_i, \sum_i x_i \vc{\alpha}_i)
$$
(so, the colors are pairs
$(a,\vc{\beta})$, where $a\in \GF{q}$ and $\vc{\beta}\in \GF{q}^s$).
We claim that $D$ is a required perfect coloring.
Indeed, changing coordinate $i$ by $u\ne0$ changes the value of~$D$ by $(u,u\vc{\alpha}_i)$. For each fixed $u$ these cover all second components once.
We see that the neighbors of a vertex of color~$(a,\vc{\beta})$
have, exactly once each, colors $(a+u,\vc{\beta}+\vc{\gamma})$, $0\ne u\in \GF{q}$, $\vc{\gamma}\in \GF{q}^s$,
i.e., all $(q-1)M$ colors that differ from~$(a,\vc{\beta})$ in the first component.
It remains to relabel the colors
$(a,\vc{\beta})$ by $qi+j$ with $i$ encoding $\vc{\beta}$ and $j$ encoding~$a$.
\end{proof}

The support $D^{-1}((0,\vc{0}))$ of the color $(0,\vc{0})$ from the proof is known as a generalized Reed--Muller code \cite[\S 5.4]{AssmKey92} $\mathcal{R}_q(r,s)$ of order~$r=(q-1)s-2$ (with the convention $\mathcal{R}_q(-1,s)=\{\vc{0}\}$
for the special case $(q,s)=(2,1)$).
Codes with such parameters (RM-like codes) and partitions
into such codes (essentially, perfect colorings with parameters as in Lemma~\ref{l:RM}) are studied in~\cite{Romanov:2022}, \cite{KrotovShi:3ary}.
Existence of such codes for non-prime-power~$q$ is an open problem related 
with the problem of existence of $1$-perfect code. For prime-$p$ power $q$ and $n$ being a power of~$p$ (not necessarily a power of~$q$), such codes and partitions can also exist, {see \cite{BesKro:Ch6}}.

\subsection{Construction}

\begin{theorem}\label{th:rec}
Assume that there exists 
a uniform collection of~$M$
perfect $r$-colorings of~$H(n,q)$
with quotient matrix~$S$, 
and assume that the first 
of those colorings essentially depends
on all arguments.
Assume that there exists a perfect 
$Mq$-coloring~$E$ of $H(M,q)$ with 
quotient matrix~$T$, 
see~\eqref{eq:RM}.
Then there exists 
a uniform collection 
of~$M$ 
perfect $r$-colorings of~$H(qn+M,q)$
with quotient matrix~$S+(q-1)^2nI + (q-1)MP$,
where~$I$ is the identity matrix and~\textcolor{black}{$P$}
is the square matrix
whose each row is the 
density vector
corresponding to~$S$,
such that each coloring
essentially depends
on all arguments.
\end{theorem}
\begin{proof}
Let $C_0$, \ldots, $C_{M-1}$ be 
the uniform collection
from the hy\-po\-the\-sis of the theorem.
For each $i \in \{0,\ldots,M-1\}$
and  $j \in \{0,\ldots,q-1\}$,
define
\begin{equation}\label{eq:D}
 D_{q\cdot i+j}(\vc{x}^0,
 \ldots,
 \vc{x}^{q-1})
 := C_i (\vc{x}^j),
\qquad \vc{x}^l \in \Sigma^n.
\end{equation}
That is, $D_{q\cdot i+j}$
is a perfect coloring
of $H(qn,q)$ with quotient matrix
$S+(q-1)^2nI$ obtained from~$C_i$
by adding $(q-1)n$ inessential arguments, 
$jn$ in the left and $(q-1-j)n$ in the right.
Now define
$$
F(y_0,\ldots,y_{M-1},x_0,\ldots,x_{qn-1}) :=
D_{E(y_0,\ldots,y_{M-1})}(x_0,\ldots,x_{qn-1}).
$$

We first claim that $F$
depends essentially 
on all $qn+M$ arguments.
Indeed, since $D_{j}$
depends essentially
on the $j$th group
of $n$ arguments,
$j=0,\ldots,q-1$,
we see that $F$ depends
essentially 
on each of the last 
$qn$ arguments.
Next, consider a tuple 
$\vc{y}'=(y'_0, \ldots, y'_{M-1})$ 
such that
$E(\vc{y}')=0$.
From the parameters of
the perfect coloring~$E$,
for every~$\vc{y}''$
adjacent to~$\vc{y}'$ in $H(M,q)$
we have 
$E(\vc{y}'')=qi''+j''\not\equiv 0\bmod q$.
Therefore, substituting
such~$\vc{y}''$ for
the first~$M$ arguments
of~$F$,
we obtain 
$D_{qi''+j''}$ that does not
depend essentially 
on the first~$n$
arguments $x_0$, \ldots, $x_{n-1}$.
On the other hand,
by substituting~$\vc{y}'$
we obtain 
$D_{0}$, which depends
essentially 
on the first~$n$
arguments.
It follows that~$F$
essentially depends
on each of the first~$M$
arguments.

Next, we claim that $F$
is a perfect coloring
with required quotient matrix.

For each vertex $\vc{z}=(\vc{y},\vc{x})=(y_0,\ldots,y_{M-1},x_0,\ldots,x_{qn-1})$
of color~$k=F(\vc{z})$,
we separate its neighbors
into two types.
Every neighbor of $1$st
type differs from $\vc{z}$
in one of the
last~$qn$ positions. 
Since 
$D_{E(\vc{y})}$ 
is a perfect coloring 
with quotient matrix 
$S'':=S+(q-1)^2nI$,
the vertex $\vc{z}$ 
has exactly $S''_{k,t}$ neighbors 
of $1$st type of color~$t$.

Every neighbor of $2$nd type
differs from $\vc{z}$
in one of the first~$M$ positions.
Let $E(\vc{y}) = qi+j$, 
$j\in\{0,...,q-1\}$, 
$i\in\{0,...,M-1\}$.
From the quotient
matrix of~$E$,
we see that
$\vc{y}$ has exactly
one neighbor~$\vc{y}'$
such that
$E(\vc{y}') = qi'+j'$,
for all
$j'\in\{0,...,q-1\}\setminus\{j\}$,
$i'\in\{0,...,M-1\}$,
and no other neighbors.
Now we see that 
the multiset of colors
of the neighbors of $2$nd type is 
$$ 
\biguplus_{j'\in\{0,...,q-1\}\setminus\{j\}}
\{D_{q\cdot i'+j'}(\vc{x})\}_{i'=0}^{M-1}
= \biguplus_{j'\in\{0,...,q-1\}\setminus\{j\}}
\{C_{i'}(x_{nj'},...,x_{nj'+n-1})\}_{i'=0}^{M-1}.
$$
Since $(C_0,\ldots,C_{M-1})$ is a uniform collection,
each multiset under the union contains every color~$t$
with multiplicity~$\rho_tM$, by Lemma~\ref{l:dens}.
After summing we get
multiplicity~$(q-1)\rho_{t} M$,
where $\rho_{t}$ is the $(k,t)$th element
$P_{k,t}$ of the matrix~$P$
for every~$k$.

Finally, we see that $F$
is a perfect coloring with
quotient 
matrix~$S+(q-1)^2nI + (q-1)MP$.
It remains to construct
a uniform collection.
Denote
$F_{C_0,\ldots,C_{M-1}}=F$,
including in this notation
the dependence of~$F$ on
the input collection
of colorings. Now we cyclically
permute the colorings from the
input collection  by
a cyclic permutation~$\pi$ 
of~$\{0,\ldots,M-1\}$.
Denote 
$F_s:=F_{C_{\pi^s(0)},\ldots,C_{\pi^s(M-1)}}$. 
It is straightforward 
to see that
$\{F_s\}_{s=0}^{M-1}$
is a uniform collection.
Each~$F_s$ essentially depends on all arguments by the same
argument as for~$F$.
\end{proof}
\begin{remark}
It is possible to reduce the number 
of arguments in the resulting perfect coloring in 
Theorem~\ref{th:rec} by reducing the number 
of inessential arguments 
in the definition~\eqref{eq:D}
of~$D_{qi+j}$. 
However in general, 
for a given off-diagonal part of 
the quotient matrix,
the existence 
of perfect colorings in $H(n',q)$
and $H(n''',q)$ without non-essential arguments
does not guarantee the same for $H(n'',q)$
where $n'<n''<n'''$.
The easiest counterexample is
the quotient matrix 
$\begin{pmatrix}
 n-2 & 2 \\ 2 & n-2  
\end{pmatrix}$, $q=2$, $n=2,3,4$.

\end{remark}
Applying the construction in Theorem~\ref{th:rec}
recursively, we get the following.
\begin{corollary}\label{c:constr}
    Under the hypothesis and notation of Theorem~\ref{th:rec},
    for every nonnegative integer~$i$
    there exists a perfect coloring of 
    $H(n_i,q)$, where 
    $$n_i=\Big(n-\frac{M}{1-q}\Big)q^i+\frac{M}{1-q},$$
    with quotient matrix
    $$ S + (n_i-n-iM)(q-1)I + i(q-1)MP$$
    such that the coloring 
    essentially depends on all arguments.
\end{corollary}
The diagonal part of the quotient matrix
grows exponentially in~$i$, which is seen from 
the term $(n_i-n-iM)(q-1)I$ and the formula 
for~$n_i$, 
while the off-diagonal part grows linearly,
which is seen from 
the term $i(q-1)MP$.

\begin{proof}
Let $\vc{\rho}$ be the density vector of the initial collection.
Since $\vc{\rho} S=n(q-1)\vc{\rho}$ and $\vc{\rho} P=\vc{\rho}$, we have
\[
\vc{\rho}\bigl(S+(q-1)^2nI+(q-1)MP\bigr)
=(qn+M)(q-1)\vc{\rho}.
\]
Thus the output colorings have the same density vector~$\vc{\rho}$,
and hence the same matrix~$P$ is used at every iteration.

Write $S_i$ for the quotient matrix after $i$ iterations.
By Theorem~\ref{th:rec},
\[
n_{i+1}=qn_i+M,\qquad
S_{i+1}=S_i+(q-1)^2n_iI+(q-1)MP,
\]
with $n_0=n$ and $S_0=S$.
Solving the first recurrence gives the stated formula for~$n_i$.
Moreover,
\[
(q-1)\sum_{j=0}^{i-1}n_j=n_i-n-iM,
\]
so summing the second recurrence gives
\[
S_i=S+(n_i-n-iM)(q-1)I+i(q-1)MP.
\]
The required uniform collection, with every coloring essentially
depending on all arguments, exists at each step by
Theorem~\ref{th:rec}.
\end{proof}


We now apply the main construction to obtain perfect $2$-colorings of $H(N,q)$ with prescribed off-diagonal entries of the quotient matrix and all arguments essential. The basic construction uses a partition of the space into $1$-perfect codes (cosets of the Hamming code). Under an additional restriction on the parameters, a slightly larger length~$N$ can be obtained in the binary case using a partition into extended $1$-perfect codes. A further development of the technique yields an improvement in the nonbinary case as well.
\begin{theorem}\label{th:bc}
Let $q$ be a prime power.
Suppose that positive integers $b$ and~$c$ have a common divisor~$e$
such that $M:=(b+c)/e$ is a power of~$q$,
and choose~$e$ as the largest such divisor.
Write $e=i(q-1)+t$, where $i\ge0$ and $0\le t<q-1$.
\begin{itemize}
\item[\rm(a)] There exists a perfect $2$-coloring
of $H(N,q)$
without non-essential arguments
and with quotient matrix
$\begin{pmatrix}
N(q-1)-b & b \\ c & N(q-1)-c
\end{pmatrix}$,
where
\[
N=
\begin{cases}
\displaystyle
\frac{((t+1)M-t)q^i-M}{q-1},
& t>0,\\[6pt]
\displaystyle
\frac{(qM-q+1)q^{i-1}-M}{q-1},
& t=0.
\end{cases}
\]
In particular, for $q=2$,
$$N=(2M-1)2^{e-1}-M.$$

\item[\rm(b)] If $q=2$ and $M\ge8$,
then the same holds with
$$N=M\cdot2^e-M.$$

\item[\rm(c)] If $q>2$ and $e\ge q-1$,
then the same holds with
\[
N=\frac{((q+t)M-t)q^{i-1}-M}{q-1}.
\]
\end{itemize}
\end{theorem}
For example, taking $b=3c$ in part~(a),
we obtain, for every positive integer~$c$,
a perfect $2$-coloring of $H(N,2)$,
where $N=7\cdot2^{c-1}-4$,
with quotient matrix
\[
\begin{pmatrix}
N-3c & 3c\\
c & N-c
\end{pmatrix},
\]
with all arguments essential and the first color class
having density~$1/4$.
For $c=2$, this gives a coloring of $H(10,2)$
with quotient matrix $\Quot{4&6\\2&8}$.
\begin{proof}
(a)
Put $b'=b/e$, $c'=c/e$, and $n=(M-1)/(q-1)$.
First suppose that $t>0$.
By Corollary~\ref{c:constr},
where
the initial length is~$tn$, the number of iterations is~$i$,
$n_i=N$, and
$MP=\begin{pmatrix}
    c' & b' \\ c' & b'
\end{pmatrix}$,
it is sufficient
to construct
a uniform collection
of~$M$ perfect $2$-colorings
of~$H(tn,q)$ with quotient matrix
$t\begin{pmatrix}
    c'-1 & b' \\ c' & b'-1
\end{pmatrix}$
without inessential arguments.

To do this, we consider
the partition of the vertex set of $H(n,q)$ into $M$ cosets
$H_0$, \ldots, $H_{M-1}$
of the $q$-ary Hamming code of length~$n$,
which is a linear subspace
of the vector space $\GF{q}^n$
associated to
the vertex set of $H(n,q)$
and at the same time
a $1$-perfect code,
i.e., its characteristic
function is a perfect coloring
with quotient matrix
$\begin{pmatrix}
    M-2 & 1 \\ M-1 & 0
\end{pmatrix}$.
Define the $2$-coloring
$C_i$, $i=0,\ldots,M-1$,
as the characteristic
function of the union
$H^*_i=H_i\cup H_{i+1}\cup\ldots
\cup H_{i+b'-1}$ of $b'$ cosets,
where the indices are
calculated modulo~$M$.
Since every vertex is contained
in exactly~$b'$ sets from~$H^*_i$,
$i=0,\ldots,M-1$, the collection
$(C_0,\ldots,C_{M-1})$ is uniform.
It is easy to see that
every coloring~$C_i$ is perfect
with quotient matrix
$\begin{pmatrix}
    c'-1 & b' \\ c' & b'-1
\end{pmatrix}$
(we can
use the second part of
Corollary~\ref{l:union}
to see that $H_i$, \ldots, $H_{i+b'-1}$ induce
a $(b'+1)$-coloring,
then the first part to see that their union induces
a $2$-coloring).

It remains to note that if
$C_i$ has an inessential argument,
then the cosets
$H_0$, \ldots, $H_{M-1}$
are grouped in sets of size~$q$,
where the vertices
of all cosets in each set are colored
with the same color.
Indeed, adding all possible field elements in that coordinate
gives $q$ distinct cosets, since the Hamming code
contains no word of weight~$1$.
This is impossible because $q$ does not divide~$b'$, according to the choice of~$e$.
Hence all arguments of the coloring~$C_i$ are essential.

Finally, to multiply the quotient matrix by~$t$, define
\[
\widetilde C_i(\vc{x}^0,\ldots,\vc{x}^{t-1})
=
C_i(\vc{x}^0+\cdots+\vc{x}^{t-1}),
\qquad \vc{x}^j\in\GF{q}^n,
\]
where addition is component-wise.
The collection
$(\widetilde C_0,\ldots,\widetilde C_{M-1})$
is uniform, and every coloring~$\widetilde C_i$
has all arguments essential.
For every fixed input, each neighbor of the sum
$\vc{x}^0+\cdots+\vc{x}^{t-1}$
is obtained by changing a coordinate
in exactly~$t$ ways, one in each block.
Therefore every coloring~$\widetilde C_i$
is perfect with quotient matrix
$t\begin{pmatrix}
    c'-1 & b' \\ c' & b'-1
\end{pmatrix}$,
as required.

If $t=0$, use the same construction with $q-1$ blocks
in place of~$t$ and $i-1$ iterations.
The off-diagonal entries are then
$(q-1+(i-1)(q-1))b'=b$ and
$(q-1+(i-1)(q-1))c'=c$,
and the length is
\[
N=\left((q-1)n+\frac{M}{q-1}\right)q^{i-1}
-\frac{M}{q-1}
=\frac{(qM-q+1)q^{i-1}-M}{q-1},
\]
as required.

 (b) For $q=2$ and $M\ge 8$, we consider
 two extended Hamming codes $C$ and~$C'$ in $H(M,2)$ whose linear span coincides with the set of all even-weight
 words of $H(M,2)$ (such two codes exist by \cite[Theorem 3.4]{EtzionVardy:1998}). For every odd-weight coset~$C''$ of~$C'$,
the set $C+C''$ consists of all odd-weight words
of~$H(M,2)$. In particular,
 \begin{itemize}
  \item[(*)] \emph{For every $i\in\{0,\ldots,M-1\}$, there are
  $\vc{x}\in C$ and $\vc{x}' \in C''$ differing only
  in the $i$th position.}
 \end{itemize}
 We partition the vertex set
 into~$M$ even-weight cosets
 $C=C_0$, $C_1$, \ldots, $C_{M-1}$
 of~$C$
 and~$M$ odd-weight cosets
 $C'_0$, \ldots, $C'_{M-1}$
 of~$C'$.
 Coloring $c'$ cosets of~$C$ and
 $c'$ cosets of~$C'$ by color~$0$
 and the other cosets by color~$1$,
 we get a perfect coloring
 with quotient matrix
 $\begin{pmatrix}
     c' & b' \\ c' & b'
 \end{pmatrix}$
 (this follows from the following
 property of the extended Hamming code:
 every odd-weight vertex has
 exactly one neighbor from the code).
 There is an odd-weight coset
 of~$C'$ whose color is different from the color of~$C$; hence by (*)
 all arguments are essential.
 By cyclic permutation on the indices
 of the cosets, we obtain
 a uniform collection of colorings.
 Applying Corollary~\ref{c:constr}
 completes the proof.

 (c)
Put $b'=b/e$, $c'=c/e$, $M=q^m$, and
$n=(M-1)/(q-1)$.
We first construct a uniform collection
of~$M$ perfect $2$-colorings of~$H(M+tn,q)$
with quotient matrix
\[
(q-1+t)\begin{pmatrix}
c'&b'\\c'&b'
\end{pmatrix}-tI
\]
without inessential arguments.

Let $\vc{\alpha}_0,\ldots,\vc{\alpha}_{M-1}$
be all elements of~$\GF{q}^m$,
and let $H$ be a check matrix
of the $q$-ary Hamming code of length~$n$.
For a word
$(\vc{y},\vc{z}^{\,1},\ldots,\vc{z}^{\,t})$,
where $\vc{y}\in\GF{q}^M$ and
$\vc{z}^{\,l}\in\GF{q}^n$,
define
\[
a=\sum_{j=0}^{M-1}y_j,
\qquad
\vc{\beta}=
\sum_{j=0}^{M-1}y_j\vc{\alpha}_j
+\sum_{l=1}^t H\vc{z}^{\,l}.
\]
Choose two distinct subsets $B,B'\subset\GF{q}^m$
of size~$c'$.
Color the word by~$0$ if
$a=0$ and $\vc{\beta}\in B'$, or
$a\ne0$ and $\vc{\beta}\in B$;
otherwise color it by~$1$.

We claim that this coloring is perfect
with the required quotient matrix.
Indeed, changing $y_j$ by $u\ne0$
changes $(a,\vc{\beta})$ by
$(u,u\vc{\alpha}_j)$.
Thus the neighbors obtained by changing
the first~$M$ arguments have, exactly once each,
all pairs $(a',\vc{\beta}')$ with $a'\ne a$.
Among them, exactly $(q-1)c'$ have color~$0$
and $(q-1)b'$ have color~$1$.

On the other hand, changing an argument
in any fixed block~$\vc{z}^{\,l}$
preserves~$a$ and gives, exactly once each,
all second components different from~$\vc{\beta}$.
Consequently, each such block contributes
\[
\begin{pmatrix}
c'-1&b'\\c'&b'-1
\end{pmatrix}
\]
to the quotient matrix.
Summing these contributions gives the required matrix.

Next, we verify that all arguments are essential.
Suppose that the argument~$y_j$ is inessential.
For every $u\ne0$, choose $a\notin\{0,-u\}$,
which is possible since $q>2$.
Both groups indexed by $a$ and $a+u$
use the set~$B$, so
$B+u\vc{\alpha}_j=B$.
The transition from the group indexed by~$0$
to the group indexed by~$u$ also gives
$B'+u\vc{\alpha}_j=B$.
Hence $B'=B$, a contradiction.

Suppose now that an argument in a block~$\vc{z}^{\,l}$
is inessential, and let $\vc{h}\ne\vc{0}$
be the corresponding column of~$H$.
Then $B+u\vc{h}=B$ for every $u\in\GF{q}$.
Therefore $B$ is a union of cosets
of the subgroup $\{u\vc{h}:\,u\in\GF{q}\}$,
and $q$ divides~$c'$.
But then $b'+c'=M$ implies that $q$ also divides~$b'$,
contrary to the maximal choice of~$e$.
Thus all arguments are essential.

Replacing $B$, $B'$ simultaneously by
$B+\vc{\gamma}$, $B'+\vc{\gamma}$,
for all $\vc{\gamma}\in\GF{q}^m$,
gives a uniform collection of~$M$ such colorings:
each word has color~$0$ in exactly~$c'$
members of the collection.

Applying Corollary~\ref{c:constr}
with initial length $M+tn$ and $i-1$ iterations,
we obtain a perfect coloring with all arguments essential
and off-diagonal entries
\[
\bigl(q-1+t+(i-1)(q-1)\bigr)b'=b,
\qquad
\bigl(q-1+t+(i-1)(q-1)\bigr)c'=c.
\]
Its length is
\[
N=\left(M+tn+\frac{M}{q-1}\right)q^{i-1}
-\frac{M}{q-1}
=\frac{((q+t)M-t)q^{i-1}-M}{q-1},
\]
as required.
 \color{black}
\end{proof}

\color{black}
\begin{corollary}\label{c:bu}
Let $q$ be a prime power and $\rho=r/M$,
where $M$ is a power of~$q$, $0<r<M$, and $q\nmid r$.
For every positive integer~$e$, write
$e=i(q-1)+t$, where $i\ge0$ and $0\le t<q-1$.
Then there exists a $\{0,1\}$-valued function
on the vertices of $H(n,q)$
of degree~$e M/q$,
with all arguments essential
and density~$\rho$ of ones, where
\[
n=
\begin{cases}
(2 M-1)2^{e-1}-M,
& q=2,\ M<8,\\[3pt]
M2^e-M,
& q=2,\ M\ge8,\\[3pt]
\displaystyle\frac{e(M-1)}{q-1},
& q>2,\ e<q-1,\\[7pt]
\displaystyle\frac{((q+t)M-t)q^{i-1}-M}{q-1},
& q>2,\ e\ge q-1.
\end{cases}
\]
\end{corollary}

\begin{proof}
Apply Theorem~\ref{th:bc}
with $b=er$ and $c=e(M-r)$.
Since $q\nmid r$, the largest common divisor~$e'$
of $b$ and $c$ for which $(b+c)/e'$ is a power of~$q$
is precisely~$e$.
The stated lengths follow from
parts~(a), (b), (a), and~(c), respectively.
The resulting coloring has density
$b/(b+c)=r/M$ of ones
and nontrivial eigenvalue
$n(q-1)-b-c=n(q-1)-e M$,
so its degree is~$e M/q$.
\end{proof}
\color{black}

%

\section{A \texorpdfstring{$q$}{q}-ary Simon--Wegener theorem}\label{s:sensitivity}

For a coloring~$f$ of~$H(n,q)$,
 the \emph{sensitivity}~\cite{AsensioSensitivity} is defined as follows:
\[
 s(f,\vc{x})=\bigl|\{\vc{y}\in N(\vc{x}):f(\vc{y})\ne f(\vc{x})\}\bigr|,
 \qquad s(f)=\max_{\vc{x}}s(f,\vc{x}).
\]
Note that for a perfect coloring $f$, $s(f)$ is the maximum (over all rows) of the sum of the off-diagonal elements in a row of the quotient matrix.
For $q=2$ and two colors, the classical results of Simon~\cite{Simon} (the lower bound)
and Wegener~\cite{Wegener1985,Wegener} (the upper bound) give the asymptotics of the
minimum of $s(f)$ over functions with $n$ essential arguments.

\begin{theorem}\label{th:sensitivity}
Fix an integer $q\ge 2$, and let $s_q(n)$ be the minimum of $s(f)$
over all colorings of $H(n,q)$ with $n$ essential arguments.
As $n\to\infty$,
\begin{multline*}
 \frac{q-1}{2}\log_q n-\frac{q-1}{2}\log_q\log_q n-O_q(1)
 \le s_q(n)\\
 \le \frac{q-1}{2}\log_q n+\frac{q-1}{4}\log_q\log_q n+O_q(1).
\end{multline*}
In particular, $\displaystyle s_q(n) = \frac{q-1}{2}\log_q n + O_q(\log\log n) $.
The upper bound is attained by a $2$-coloring.

\end{theorem}
The proof follows the ideas of Simon~\cite{Simon}
for the lower bound and Wegener~\cite{Wegener1985}
for the upper bound.

\begin{proof}
\emph{The lower bound.} Let $f$ be a coloring with $n$ essential
arguments and let $s=s(f)$. For each coordinate~$i$, choose symbols
$a_i,b_i$ so that the set
\def\insrt#1#2#3{#1_{#2,#3}}
\[
 D_i=\{\vc{z}:\ f(\insrt{\vc{z}}{i}{a_i})
                 \ne f(\insrt{\vc{z}}{i}{b_i})\}
\]
is nonempty. Here $\vc{z}$ is a word on the remaining $n-1$
coordinates and $\insrt{\vc{z}}{i}{a}$ denotes the word obtained by
inserting $a$ in position~$i$.
For $\vc{z}\in D_i$, the sum $s(f,\insrt{\vc{z}}{i}{a_i})+s(f,\insrt{\vc{z}}{i}{b_i})$ includes at least $q$ contributions from direction~$i$
(indeed, $f(\insrt{\vc{z}}{i}{a_i})\ne f(\insrt{\vc{z}}{i}{b_i})$ contributes twice, and for each of
the other $q-2$ symbols $a\not\in\{a_i,b_i\}$ we have $f(\insrt{\vc{z}}{i}{a})\ne f(\insrt{\vc{z}}{i}{a_i})$ or $f(\insrt{\vc{z}}{i}{a})\ne f(\insrt{\vc{z}}{i}{b_i})$). Also, for every
neighbor $\vc{z}'\notin D_i$ of $\vc{z}$ in $H(n-1,q)$, at least one
of the two corresponding edges in $H(n,q)$ changes color.
Thus, the external degree of each vertex~$\vc{z}$ of~$D_i$ in $H(n-1,q)$ is at most $s(f,\insrt{\vc{z}}{i}{a_i})+s(f,\insrt{\vc{z}}{i}{b_i})-q\le 2s -q$
and the internal degree is at least $(n-1)(q-1)-(2s-q)$.

The edge-isoperimetric bound for $H(m,q)$ gives an average internal
degree of at most $(q-1)\log_q|D|$ for a nonempty vertex set~$D$
\cite{SquierTorrenceVogt}. In particular,
$ (n-1)
 (q-1)
 -(2s-q)\le (q-1) \log_q{|D_i|}
$; equivalently,
\begin{equation}\label{eq:Di}
q^{-(2s-q)/(q-1)}\le\frac{|D_i|}{q^{n-1}} .
\end{equation}
For each word~$\vc{x}$, the number of coordinates~$i$
for which deleting the $i$th coordinate gives a word in~$D_i$
is at most $s(f,\vc{x})$:
each such coordinate contributes at least one neighbor
of a different color.
Averaging over all words~$\vc{x}\in\Sigma^n$, we obtain
\begin{equation}\label{eq:sensitivity-lower}
s
\ge
\frac{1}{q^n}\sum_{\vc{x}\in\Sigma^n}s(f,\vc{x})
\ge
\sum_{i=0}^{n-1}\frac{|D_i|}{q^{n-1}}
\stackrel{\eqref{eq:Di}}\ge
nq^{-(2s-q)/(q-1)}.
\end{equation}
Taking the logarithm, we get $s\ge \frac{q-1}{2}(\log_q n-\log_q s)+\frac q2$.
If $s\ge \frac{q-1}{2}\log_q n$, the stated lower bound follows
directly. Otherwise $\log_q s\le \log_q\log_q n+O_q(1)$,
which gives the same conclusion.

\medskip
\emph{The upper bound.}
We divide the $n$ coordinates of $H(n,q)$ into two parts:
the first $k$ coordinates (the exact value of $k$ will be specified later) are called \emph{address coordinates}, and the corresponding $k$-words are called \emph{address words};
the remaining $n-k$ coordinates are called \emph{data coordinates}.
We use the alphabet $\Sigma=\{0,\ldots,q-1\}$.
For an address word $\vc{z}=(z_0,\ldots,z_{k-1})$,
we define its weight
\[
\mathrm{wt}(\vc{z}):=z_0+\cdots+z_{k-1}\in\{0,\ldots,k(q-1)\},
\]
where the sum is over $\mathbb{Z}$.
Put $L:=k(q-1)$ and denote by $A_k$ the set of address words of weight~$\lfloor L/2\rfloor$.
We choose the smallest $ k$ such that $|A_k|\ge n-k$.
Assign to each data coordinate $i\in\{k,\ldots,n-1\}$ a distinct address $\vc{a}_i\in A_k$.

\begin{itemize}
\item[(*)]
\emph{We claim that $k\le\log_q n+\tfrac12\log_q\log_q n+O_q(1)$.}

Indeed, the number of address words of weight~$w$ equals the coefficient of $z^w$ in the polynomial $(1+z+\cdots+z^{q-1})^k$.
The asymptotic formula~\cite{Eger:2014} for $w=\lfloor L/2\rfloor$ gives
$
|A_k|\sim q^k\sqrt{{6}/{\pi k(q^2-1)}}
$,
and hence $|A_k|\ge c_q q^k/\sqrt{k}$ for some constant $c_q>0$ and all sufficiently large~$k$.
Set
\[
K=\left\lceil\log_q n+\tfrac12\log_q\log_q n+C_q\right\rceil.
\]
Since $K/\log_q n\to1$, choosing the constant~$C_q$ sufficiently large ensures that, for all sufficiently large~$n$,
\[
|A_K|\ge c_q\frac{q^K}{\sqrt K}
\ge c_q q^{C_q}n\sqrt{\frac{\log_q n}{K}}\ge n\ge {n}-K.
\]
So, $k\le K$.
This proves~(*).
\end{itemize}

Now we define the function:
\begin{equation}\label{eq:cases}
f(\vc{z},x_k,\ldots,x_{n-1})=
\begin{cases}
1 & \text{if $\mathrm{wt}(\vc{z})>\lfloor L/2\rfloor$,}\\
\lfloor 2x_i/q\rfloor & \text{if $\vc{z}=\vc{a}_i$ for some $i\in\{k,\ldots,n-1\}$,}\\
0 & \text{otherwise.}
\end{cases}
\end{equation}

\begin{itemize}
\item[(**)]
\emph{We claim that}
\[
s(f)\le\lceil L/2\rceil+\lceil q/2\rceil
\le\frac{q-1}{2}\log_q n
+\frac{q-1}{4}\log_q\log_q n+O_q(1).
\]

Indeed, consider a word $(\vc{z},\vc{x})$ and put $w=\mathrm{wt}(\vc{z})$.
There are exactly $w$ neighbors whose address weight is smaller than~$w$, and exactly $L-w$ whose address weight is larger than~$w$.

If $f(\vc{z},\vc{x})=0$, then $w\le \lfloor L/2\rfloor$.
All neighbors with smaller address weight have color~$0$.
If $w<\lfloor L/2\rfloor$, the data coordinates do not affect the color.
If there is an address neighbor of color~$1$, its address weight is at least~$\lfloor L/2\rfloor$.
The coordinate changed along this edge therefore also gives $\lfloor L/2\rfloor-w-1$ neighbors with address weights $w+1,\ldots,\lfloor L/2\rfloor-1$, all of color~$0$.
Consequently, the number of neighbors of color~$1$ is at most
\[
(L-w)-(\lfloor L/2\rfloor-w-1)=\lceil L/2\rceil+1.
\]
If there is no address neighbor of color~$1$, the sensitivity is zero, so the same bound holds.
If $w=\lfloor L/2\rfloor$, at most $L-\lfloor L/2\rfloor=\lceil L/2\rceil$ address neighbors have color~$1$.
Data neighbors can have color~$1$ only when $\vc{z}=\vc{a}_i$, and then only by changing $x_i$ to one of the $\lfloor q/2\rfloor$ symbols at least~$\lceil q/2\rceil$ (according to the second line of~\eqref{eq:cases}).
Thus in this case $s(f,(\vc{z},\vc{x}))\le\lceil L/2\rceil+\lfloor q/2\rfloor$.

If $f(\vc{z},\vc{x})=1$, then $w\ge \lfloor L/2\rfloor$, and similar arguments give $s(f,(\vc{z},\vc{x}))\le\lfloor L/2\rfloor+\lceil q/2\rceil$.
In all cases we have
$s(f)\le\lceil L/2\rceil+\lceil q/2\rceil$.
The second inequality in~(**) follows from~(*) and $L=k(q-1)$.
This proves~(**).
\end{itemize}

It remains to note that
\begin{itemize}
\item[(***)]
\emph{All arguments of $f$ are essential.}

Indeed, each data coordinate~$i$ is essential at its assigned address~$\vc{a}_i$, since changing $x_i$ from $0$ to $q-1$ changes the value of~$f$ from $0$ to~$1$.
To verify essentiality of an address coordinate~$j$, take any assigned address~$\vc{a}_i$ and change its $j$th symbol to any other symbol.
The new address weight differs from~$\lfloor L/2\rfloor$.
If it is smaller than~$\lfloor L/2\rfloor$, choose $x_i=q-1$; if it is larger than~$\lfloor L/2\rfloor$, choose $x_i=0$.
In either case the change in address coordinate~$j$ changes the value of~$f$.
This proves~(***) and completes the proof of the upper bound.
\end{itemize}

\end{proof}

\begin{remark}
 The same asymptotics hold if one prescribes any fixed number
$r\ge 3$ of colors. To see this, take a two-coloring on $n-\ell$
coordinates from the construction, with $\ell=\lceil\log_q(r-1)\rceil$,
and append $\ell$ coordinates. Keep the value $0$ on the original
zero set. On the original one set, assign distinct colors
$2,\ldots,r-1$ to $r-2$ distinct words in the new coordinates,
and color the remaining words $1$. Each old essential coordinate
remains essential, since an original change from $0$ to $1$ still
changes color for every fixed word in the new coordinates. Each new
coordinate is essential, since color $2$ occurs at just one word
in those coordinates on the original one set.
The sensitivity increases by at most $\ell(q-1)$.
\end{remark}


{For perfect $2$-colorings of $H(n,2)$ with symmetric quotient
matrix, $s(f)=b=c$, and Proposition~\ref{p:Tar} gives
$s_{2,1/2}(m) \sim  \log_2 m$,
where $s_{2,1/2}(m)$ is the minimum sensitivity over all
balanced ($b=c$) perfect $2$-colorings of $H(n,2)$
with at least~$m$ essential arguments. Note that, in contrast to arbitrary colorings, for perfect colorings we cannot write $m=n$ because the construction gives an example for special infinite sequence of values of~$m$, not for all positive integers.} The following corollary of our main construction generalizes this result for perfect
$2$-colorings of $H(n,q)$ with different densities.

\begin{corollary}\label{c:perfect-sensitivity}
Let $q$ be a prime power, and let $0<\rho<1$ be such that
$M\rho$ is an integer for some power~$M$ of~$q$.
Choose the smallest such~$M$ and put
$R=\max\{\rho,1-\rho\}$.

Let $s_{q,\rho}(m)$ be the minimum sensitivity over all
perfect $2$-colorings of $H(n,q)$, where $n>m$,
with at least $m$ essential arguments and density~$\rho$
of color~$1$.
Then, as $m\to\infty$,
\[
 \frac{qR}{2-\log_q\!{(2q{-}2)}}\log_q m
 -O_{q,\rho}(\log\log m)
 \le s_{q,\rho}(m)
 \le (q-1)MR\log_q m+O_{q,\rho}(1).
\]
For $q=2$, the lower bound can be strengthened to
$$ \displaystyle
  2R\log_2 m-O_\rho(1)\le s_{2,\rho}(m).
$$
\end{corollary}
\begin{proof}
Put $r=M\rho$.
Partition the vertices of
$H((M-1)/(q-1),q)$ into the $M$ cosets of a $q$-ary
Hamming code.
Taking the characteristic functions of the unions of
$r$ cyclically consecutive cosets gives a uniform collection
of $M$ perfect $2$-colorings with off-diagonal entries
$r$, $M-r$.
All arguments are essential: otherwise the selected cosets
would occur in groups of size~$q$, whereas $q\nmid r$
by the minimality of~$M$.

Corollary~\ref{c:constr} therefore gives colorings with
\[
 n_i=\frac{(2M-1)q^i-M}{q-1},
 \qquad
 s_i=MR\bigl(1+(q-1)i\bigr),
\]
all of whose arguments are essential.
Choosing the smallest $i$ with $n_i\ge m$ proves the upper
bound. If $n_i=m$, append an inessential argument.

For the lower bounds, let $N\ge m$ be the number of
essential arguments of a coloring under consideration.
Remove its inessential arguments and let $b$, $c$
be its off-diagonal entries, which are unchanged.
Its degree and sensitivity satisfy
\[
 d=\frac{b+c}{q},
 \qquad s=R(b+c)=qRd.
\]
For $q=2$, the bound $N=O(2^d)$ from~\cite{CHS:2020}
gives $s\ge2R\log_2 m-O_\rho(1)$.

For $q\ge3$, the bound for perfect $2$-colorings
from~\cite{Valyuzhenich:2026} gives
\[
 N\le\frac d2
 \left(\frac{q^2}{2(q-1)}\right)^d
 \qquad\text{if }d\le\frac{N+1}{2}.
\]
Taking logarithms and using $s=qRd$ yields the stated
lower bound. In the remaining case, $s>qR(N+1)/2$,
which is more than sufficient.
\end{proof}

\subsection*{Acknowledgment}

The author thanks Yuriy Tarannikov and Alexandr Valyuzhenich
for helpful discussions.
The study was funded by the Russian Science Foundation, grant 22-11-00266
(Section~\ref{s:mainconstr}),
and
within the
framework of the state contract of the Sobolev Institute of Mathematics, project FWNF-2026-0011 (Section~\ref{s:sensitivity}).



\begin{thebibliography}{99}

\bibitem{AsensioSensitivity}
S.~Asensio, Y.~Filmus, I.~Garc\'{\i}a-Marco, and K.~Knauer.
\newblock Sensitivity and {H}amming graphs.
\newblock {\em J. Graph Theory}, 112(3):296--305, 2026.
\newblock \DOI{10.1002/jgt.70034}.

\bibitem{AssmKey92}
E.~F. Assmus, Jr. and J.~D. Key.
\newblock {\em Designs and Their Codes}.
\newblock Cambridge University Press, Cambridge, 1992.
\newblock \DOI{10.1017/CBO9781316529836}.

\bibitem{BesKro:Ch6}
E.~A. Bespalov and D.~S. Krotov.
\newblock Some completely regular codes in {D}oob graphs.
\newblock In M.~Shi and P.~Sol\'e, editors, {\em Completely Regular Codes in Distance-Regular Graphs}, chapter~6. CRC Press, 2025,
pp.~379--448. \DOI{10.1201/9781003393931-6}

\bibitem{BKMTV}
E.~A. Bespalov, D.~S. Krotov, A.~A. Matiushev, A.~A. Taranenko, and K.~V. Vorob'ev.
\newblock Perfect $2$-colorings of {H}amming graphs.
\newblock {\em \href{http://onlinelibrary.wiley.com/journal/10.1002/(ISSN)1520-6610}{J. Comb. Des.}}, 29(6):367--396, June 2021.
\newblock \DOI{10.1002/jcd.21771}.

\bibitem{CHS:2020}
J.~Chiarelli, P.~Hatami, and M.~Saks.
\newblock An asymptotically tight bound on the number of relevant variables in a bounded degree {Boolean} function.
\newblock {\em \href{http://link.springer.com/journal/493}{Combinatorica}}, 40(2):237--244, 2020.
\newblock \DOI{10.1007/s00493-019-4136-7}.

\bibitem{Eger:2014}
S.~Eger.
\newblock Stirling's approximation for central extended
binomial coefficients.
\newblock {\em Amer. Math. Monthly}, 121(4):344--349, 2014.
\newblock \DOI{10.4169/amer.math.monthly.121.04.344}.

\bibitem{EtzionVardy:1998}
T.~Etzion and A.~Vardy.
\newblock On perfect codes and tilings: Problems and solutions.
\newblock {\em SIAM J. Discrete Math.}, 11(2):205--223, 1998.
\newblock \DOI{10.1137/S0895480196309171}.

\bibitem{FDF:CorImm}
D.~G.~Fon-Der-Flaass.
\newblock A bound on correlation immunity.
\newblock {\em Siberian Electronic Mathematical Reports}, 4:133--135, 2007.
\newblock \url{https://www.mathnet.ru/eng/semr149}.

\bibitem{FDF:PerfCol}
D.~G. Fon-Der-Flaass.
\newblock Perfect $2$-colorings of a hypercube.
\newblock {\em \href{http://link.springer.com/journal/11202}{Sib. Math. J.}}, 48(4):740--745, 2007.
\newblock \DOI{10.1007/s11202-007-0075-4} translated from \href{http://www.mathnet.ru/php/journal.phtml?jrnid=smj\&option_lang=eng}{Sib. Mat. Zh.} 48(4) (2007), 923-930.


\bibitem{Heden:2010:non-prime}
O.~Heden.
\newblock On perfect codes over non prime power alphabets.
\newblock In A.~A. Bruen and D.~L. Wehlau, editors, {\em Error-Correcting Codes, Finite Geometries and Cryptography}, volume 523 of {\em \href{http://www.ams.org/books/conm/}{Contemp. Math.}}, pages 173--184. AMS, Providence, RI, 2010.

\bibitem{KKM:tables}
J.~H.~Koolen,
D.~S. Krotov, and W.~H. Martin.
\newblock Completely Regular Codes:
Tables of Small Parameters
for Binary and Ternary Hamming
Graphs.
\newblock In M.~Shi and P.~Sol\'e, editors, {\em Completely Regular Codes in Distance-Regular Graphs}, chapter~7. CRC Press, 2025, pp.~449--502. \DOI{10.1201/9781003393931-7}.

\bibitem{Kro:TernDiamPerf}
D.~S. Krotov.
\newblock On diameter perfect constant-weight ternary codes.
\newblock {\em \href{http://www.sciencedirect.com/science/journal/0012365X}{Discrete Math.}}, 308(14):3104--3114, 2008.
\newblock \DOI{10.1016/j.disc.2007.08.037}.

\bibitem{KroPot:CRC&EP}
D.~S. Krotov and V.~N. Potapov.
\newblock Completely regular codes and equitable partitions.
\newblock In M.~Shi and P.~Sol\'e, editors, {\em Completely Regular Codes in Distance-Regular Graphs}, chapter~1. CRC Press, 2025, pp.~1--84. \DOI{10.1201/9781003393931-1}.

\bibitem{NisanSzegedy:1994}
\begingroup
N.~Nisan and M.~Szegedy.
\newblock On the degree of {Boolean} functions as real polynomials.
\newblock {\em Comput. Complexity}, 4(4):301--313, 1994.
\newblock \DOI{10.1007/BF01263419}.
\endgroup

\bibitem{Potapov:2025}
\begingroup
V.~N. Potapov.
\newblock On the number of relevant variables for discrete functions.
\newblock {\em Cryptogr. Commun.}, 17:989--998, 2025.
\newblock \DOI{10.1007/s12095-025-00797-4}.
\endgroup

\bibitem{Romanov:2022}
A.~M. Romanov.
\newblock On the number of q-ary quasi-perfect codes with covering radius 2.
\newblock {\em \href{http://link.springer.com/journal/10623}{Des. Codes Cryptography}}, 90(8):1713--1719, Aug. 2022.
\newblock \DOI{10.1007/s10623-022-01070-y}.

\bibitem{KrotovShi:3ary}
M.~Shi and D.~S. Krotov.
\newblock An enumeration of $1$-perfect ternary codes.
\newblock {\em \href{http://www.sciencedirect.com/science/journal/0012365X}{Discrete Math.}}, 346(7):113437(1--16), July 2023.
\newblock \DOI{10.1016/j.disc.2023.113437}.

\bibitem{Simon}
H.-U. Simon.
\newblock A tight $\Omega(\log\log n)$-bound on the time for parallel {RAM}'s to compute nondegenerated {B}oolean functions.
\newblock In {\em Foundations of Computation Theory}, LNCS~158, pp.~439--444. Springer, 1983.
\newblock \url{https://doi.org/10.1007/3-540-12689-9_124}.

\bibitem{SquierTorrenceVogt}
R.~Squier, B.~Torrence, and A.~Vogt.
\newblock The number of edges in a subgraph of a {H}amming graph.
\newblock {\em Appl. Math. Lett.}, 14(6):701--705, 2001.
\newblock \DOI{10.1016/S0893-9659(01)80030-6}.

\bibitem{Tarannikov:2002:survey}
Yu.~V. Tarannikov.
\newblock On correlation-immune and resilient {Boolean} functions.
\newblock In O.~B. Lupanov, editor,
{\em Matematicheskie Voprosy Kibernetiki},
volume~11, pages~91--148.
Fizmatlit, Moscow, 2002 (in Russian).

\bibitem{Tarannikov:2002}
\begingroup
Yu.~V. Tarannikov.
\newblock A {Simon--Wegener} type theorem for regular {Boolean} functions.
\newblock In {\em Problems of Theoretical Cybernetics}, abstracts of the XIII International Conference (Kazan, May 27--31, 2002), p.~175. Moscow, 2002 (in Russian).
\endgroup


\bibitem{Valyuzhenich:2026}
\begingroup
A.~Valyuzhenich.
\newblock An upper bound on the number of relevant variables for {Boolean} functions on the {Hamming} graph.
\newblock {\em Discrete Math.}, 349(2):114745, 2026.
\newblock \DOI{10.1016/j.disc.2025.114745}.
\endgroup

\bibitem{Valyuzhenich:2026:juntas}
\begingroup
A.~Valyuzhenich.
\newblock An upper bound on the number of relevant variables in a bounded degree {Boolean} function on the {Hamming} graph.
\newblock Preprint, 2026. \url{https://arxiv.org/abs/2609.12540}.
\endgroup

\bibitem{Wegener1985}
I.~Wegener.
\newblock The critical complexity of all (monotone) {B}oolean functions and monotone graph properties.
\newblock {\em Inform. Control}, 67(1--3):212--222, 1985.
\newblock \DOI{10.1016/S0019-9958(85)80036-X}.

\bibitem{Wegener}
I.~Wegener.
\newblock {\em The Complexity of {Boolean} Functions}.
\newblock {John} {Wiley} \& {Sons}, {Chichester} etc., 1987.

\bibitem{Wellens:2022}
J.~Wellens.
\newblock Relationships between the number of inputs and other
complexity measures of {Boolean} functions.
\newblock {\em Discrete Analysis}, 2022:19, 21~pp., 2022.
\newblock \DOI{10.19086/da.57741}.



\end{thebibliography}

\providecommand\href[2]{#2} \providecommand\url[1]{\href{#1}{#1}} \def\DOI#1{{\href{https://doi.org/#1}{https://doi.org/#1}}}\def\DOIURL#1#2{{\href{https://doi.org/#2}{https://doi.org/#1}}}

\end{document}